\documentclass[12pt,oneside]{amsart}
\usepackage[latin9]{inputenc}
\usepackage{mathrsfs}
\usepackage{amsthm}
\usepackage{amstext}
\usepackage{amssymb}

\makeatletter

\providecommand{\tabularnewline}{\\}

\numberwithin{equation}{section}
\numberwithin{figure}{section}

\title{}

\usepackage{amsfonts}\usepackage{enumerate}

\newtheorem{theorem}{Theorem}
\theoremstyle{plain}

\newtheorem{definition}{Definition}
\newtheorem{example}{Example}

\newtheorem{lemma}{Lemma}

\newtheorem{proposition}{Proposition}
\newtheorem{remark}{Remark}

\numberwithin{equation}{section}

\makeatother

\begin{document}

\title{Fuzzy Cosets and Quotient Fuzzy AG-subgroups}

\author{Amanullah$^{1}$}

\email{amanswt@gmail.com}

\author{I. Ahmad$^{1,*}$}

\email{iahmaad@hotmail.com}

\address{1. Department of Mathematics University of Malakand, Chakdara Dir
Lower, Khyber Pakhtunkhwa, Pakistan.}

\author{M. Shah$^{2}$}

\email{shahmaths\_problem@hotmail.com}

\address{2. Department of Mathematics, Government Post Graduate College Mardan,
Khyber Pakhtunkhwa, Pakistan.}

\keywords{AG-group; Fuzzy AG-subgroup; Normal fuzzy AG-subgroup; Fuzzy cosets
and fuzzy quotient AG-subgroups.\\
{*}Corresponding author }
\begin{abstract}
In this paper we extend the concept of fuzzy AG-subgroups. We introduce
some results in \emph{normal fuzzy AG-subgroups. }We\emph{ }define
\emph{fuzzy cosets} and \emph{quotient fuzzy AG-subgroups}, and prove
that the sets of their collection form an AG-subgroup and fuzzy AG-subgroup
respectively. We also introduce the fuzzy Lagrange's Theorem of AG-subgroup.
It is known that the condition $\mu(xy)=\mu(yx)$ holds for all $x,y$
in fuzzy subgroups if $\mu$ is normal, but in fuzzy AG-subgroup we
show that it holds without normality.
\end{abstract}
\maketitle

\section{Introduction}

The concept of fuzzy sets along with various operations has been introduced
by Lofti A. Zadeh in 1965 \cite{Zad}. Due to the diverse applications
ranging from engineering, computer science and social behavior studies,
the researchers have taken keen interest in the subject in its related
fields. The study of fuzzy algebraic structures was started by introducing
the concept of fuzzy subgroups by A. Rosenfeld \cite{Ros}. He formulated
the concept of fuzzy subgroup and extend the main idea of group theory
to develop the theory of fuzzy groups. Anthony and Sherwood further
redefined fuzzy groups \cite{Anto}. Many other papers on fuzzy subgroups
have also appeared which generalize various concepts of group theory
such as normal subgroups, quotient groups and cosets \cite{Meng,Wu,Wan}.

\noindent In the forty years history of AG-groupoids, though it was
explored slowly, yet in the last couple of years abundant research
was carried out in this area which attracted the attention of many
new researchers. An AG-groupoid is a generalization of commutative
semigroup. It is a nonassociative groupoid in general, in which the
left invertive law $(ab)c=(cb)a$ holds. In general, an AG-group is
a nonassociative structure in which commutativity and associativity
imply each other, and thus it becomes abelian group if any one of
them is allowed. An AG-groupoid $(G,\cdot)$ is called an AG-group
or left almost group (LA-group), if there exists a unique left identity
$e\in G$ (that is $ea=a$ for all $a\in G$ ), and for all $a\in G$
there exists $a^{^{-1}}\in G$ such that $a^{^{-1}}a=aa^{-1}=e$.
M. Kamran extended the notion of AG-groupoid to an AG-group and defined
cosets of an AG-subgroup $H$ of an AG-group $G$ and proved that
quotient $G/H$ is defined for every AG-subgroup $H$. He also proved
that Lagrange's Theorem holds for AG-group \cite{Kamran}. The third
author of this article has discussed various basic properties of AG-groups
and explored new results such as: complexes and cosets decomposition,
conjugacy relations in AG-groups, normality, normalizers and many
more \cite{Sha Thes,MshahT}. For the first time in 2003, Q. Mushtaq
and M. Khan introduced ideals in AG-groupoid and fuzzified these concepts
\cite{idealsinAG}. This attracted the attention of various other
researchers to the field of AG-groupoids and AG-groups, as a result
since then we can see lots of papers in this area. It is also worth
mentioning that various new classes of AG-groupoids have been recently
introduced \cite{SIA1,intro2,key-1-13,RAAS4,AGst5,AGst-6} and some
are just have been arXived \cite{mod,mod m} and their fuzzification
is suggested as an interesting future work.

In this paper we extend the concepts of normal fuzzy AG-subgroup \cite{I. Ahmad,key-1}.
We further define fuzzy cosets, quotient AG-subgroups and quotient
fuzzy AG-subgroups, which will provide new direction to the researchers
in this area. We also introduce a fuzzy version of the famous Lagrange's
Theorem for finite AG-groups.

\section{Preliminaries}

In this section we list some basic definitions that will frequently
be used in the subsequent sections of the paper.

\noindent A \emph{fuzzy subset} $\mu$ is a mapping $\mu:X\rightarrow[0,1]$.
The set of all fuzzy subsets of $X$ is called the \emph{fuzzy power
set} of $X$ and is denoted by $FP(X)$. Let $\mu\in FP(X)$, then
the the image of $\mu$ is a set $\{\mu(x)\,:\, x\in X\}$ and is
denoted by $\mu(X)$ or $Im(\mu).$ 

\noindent In the rest of this paper $G$ will denote an AG-group otherwise
stated and $e$ will denote the left identity of $G$.

\noindent \begin {definition} \cite{I. Ahmad} Let $\mu\in FP(G)$,
then $\mu$ is called a fuzzy AG-subgroup of $G$ if for all $x,y\in G$;

\begin{enumerate}[(i)]

\item $\mu(xy)\geq\mu(x)\wedge\mu(y)$;

\item $\mu(x^{-1})\geq\mu(x)$.

\end{enumerate}

\noindent The set of all fuzzy AG-subgroups of $G$ is denoted by
$F(G)$.

\noindent If $\mu\in F(G)$, then
\begin{equation}
\mu_{\ast}=\{x\in G\,|\text{\,}\mu(x)=\mu(e)\}.\label{a1}
\end{equation}
\end {definition}

\noindent \begin {lemma}\emph{\label{ad}\cite{I. Ahmad}} Let $\mu$
be any fuzzy AG-subgroup of $G$ i.e $\mu\in F(G)$, then for all
$x\in G$,

\begin{enumerate}[(i)]

\item$\mu(e)\geq\mu(x)$;

\item $\mu(x)=\mu(x^{-1})$.

\end{enumerate}

\noindent \end {lemma}

\noindent \begin {definition}\cite{I. Ahmad} Let $\mu\in F(G)$.
Then $\mu$ is called a \emph{normal fuzzy AG-subgroup}\textbf{ }of
$G$ if
\[
\mu(xy\cdot x^{-1})=\mu(y)\text{ }\forall x,y\in G.
\]
The set of all normal fuzzy AG-subgroups of $G$ is denoted by $NF(G)$.
\end {definition}

\noindent \begin {example} Consider an AG-group of order $4$:

\noindent \begin{center}
\begin{tabular}{l|llll}
$\cdot$  & $0$  & $1$  & $2$  & $3$\tabularnewline
\hline 
$0$  & $0$ & $1$ & $2$ & $3$\tabularnewline
$1$  & $3$ & $0$ & $1$ & $2$\tabularnewline
$2$  & $2$ & $3$ & $0$ & $1$\tabularnewline
$3$ & $1$ & $2$ & $3$ & $0$\tabularnewline
\end{tabular}
\par\end{center}

\noindent \begin{flushleft}
define fuzzy subset $\mu$ by $\mu(0)=t_{0}$ and $\mu(x)=t_{1}$
otherwise; where $t_{0},t_{1}\in[0,1]$ and $t_{0}>t_{1}$. Then $\mu$
is fuzzy AG-subgroup. However, $\mu$ is not normal fuzzy AG-subgroup.
Here $\mu_{*}=\{0\}$ which is normal fuzzy AG-subgroup.
\par\end{flushleft}

\noindent \end {example}

\noindent \begin {example} Consider an AG-group $G=<a,b:\, a^{2}=b^{2}=(ab)^{2}=e>$,
define fuzzy subset $\mu:G\rightarrow[0,1]$ by $\mu(e)=t_{0},\,\mu(a)=t_{1}$
and $\mu(x)=t_{2}$ otherwise; where $t_{0},t_{1},t_{2}\in[0,1]$
and $t_{0}>t_{1}>t_{2}$. Then $\mu$ is normal fuzzy AG-subgroup
of $G$.

\noindent \end {example}

\section{Main Results}

\begin {proposition} \label{b-1}Let $\mu\in F(G)$. Then $\mu(xy)=\mu(yx)$
$\forall x,y\in G$.\end {proposition}

\noindent \begin {proof} Let $\mu\in F(G)$, then
\begin{eqnarray*}
\mu(xy) & = & \mu(ex\cdot y)=\mu(yx\cdot e)\qquad[\mbox{using left invertive law}]\\
 & \geq & \mu(yx)\wedge\mu(e)=\mu(yx)\,\,\quad[\mu(e)\geq\mu(yx)\,\,\forall\, x,y\in G]\\
\Rightarrow\mu(xy) & \geq & \mu(yx)\quad\forall\, x,y\in G.
\end{eqnarray*}
Similarly, we can show that $\mu(yx)\geq\mu(xy)\,\,\forall\, x,y\in G$.
Thus $\mu(xy)\geq\mu(yx)\geq\mu(xy)\,\,\forall\, x,y\in G$. Hence
$\mu(yx)=\mu(xy)$. \end {proof}

\noindent \begin {lemma}\label{a}\cite[Lemma 17]{I. Ahmad} Let
$\mu$ be a fuzzy AG-subgroup of $G$. Let $x\in G$ then $\mu(xy)=\mu(y)$
$\forall\, y\in G$ if and only if $\mu(x)=\mu(e)$. \end {lemma}

\noindent \begin {theorem} Let $f$ be a homomorphism on AG-group
$G$ and $\mu$ is any normal fuzzy AG-subgroup of $f(G)$. Then $\mu\circ f\in NF(G)$.
\end {theorem}

\noindent \begin {proof} First we show that $\mu\circ f\in F(G)$.
Since
\begin{eqnarray*}
\mu\circ f(xy) & = & \mu(f(xy))\\
 & = & \mu(f(x)\cdot f(y))\\
 & \geq & \mu(f(x))\wedge\mu(f(y))\\
 & = & \mu\circ f(x)\wedge\mu\circ f(y)
\end{eqnarray*}
$\forall\, x,y\in G$, and
\begin{eqnarray*}
\mu\circ f(x^{-1}) & = & \mu(f(x^{-1}))\\
 & = & \mu((f(x))^{-1})\,\qquad\quad\\
 & = & \mu(f(x))\\
 & = & \mu\circ f(x).
\end{eqnarray*}
$\forall\, x\in G$. Hence $\mu\circ f\in F(G)$.

\noindent Next we show that $\mu\circ f\in NF(G)$, since
\begin{eqnarray*}
\mu\circ f(xy\cdot x^{-1}) & = & \mu(f(xy\cdot x^{-1}))\\
 & = & \mu(f(xy)\cdot f(x^{-1}))\\
 & = & \mu(\{f(x)f(y)\}\cdot(f(x))^{-1})\\
 & = & \mu(f(y))\,\,[\mu\in NF(f(G))]\\
 & = & \mu\circ f(y).
\end{eqnarray*}
$\forall\, x,y\in G$. Hence $\mu\circ f\in NF(G)$. \end {proof}

\noindent \begin {definition} \label{b}Let $\mu\in F(G)$, for any
$x\in G$ define a mapping
\begin{eqnarray*}
\mu_{x} & : & G\rightarrow[0,1],\,\,\mbox{by}
\end{eqnarray*}
\begin{eqnarray}
\mu_{x}(g) & = & \mu(gx^{-1})\,\,\forall g\in G.\label{eq:1a-1}
\end{eqnarray}
Then $\mu_{x}$ is called fuzzy coset of $G$ determined by $x$ and
$\mu$, and the collection of all fuzzy cosets of $\mu$ is represented
by \emph{$\mathsf{\mathit{\mathtt{\mathbb{\mathfrak{\mathscr{F}}}}}}$.}\end {definition}

\noindent In AG-groups we can define quotient AG-group by any AG-subgroup
without normality. Therefore, make use of this we can define quotient
AG-groups or factor AG-group as follows: 

\noindent \begin {theorem} \label{s}Let $\mu\in NF(G)$ and $G\diagup\mu=\{\mu_{x}\,:\, x\in G\}$.
Then $G\diagup\mu$ form an AG-group under the usual composition of
mappings define by \emph{$\mu_{x}\circ\mu_{y}=\mu_{xy}\,\,\forall\, x,y\in G$}.
\end {theorem}

\noindent \begin {proof} First we show that the composition of cosets
is well defined. Let $x,y,x_{\circ},y_{\circ}\in G$ such that $\mu_{x}=\mu_{x_{\circ}}$and
$\mu_{y}=\mu_{y_{\circ}}$.

\noindent We show that $\mu_{x}\circ\mu_{y}=\mu_{x_{\circ}}\circ\mu_{y_{\circ}}$,
i.e. $\mu_{xy}=\mu_{x_{\circ}y_{\circ}}$. Thus by (\ref{eq:1a-1})
\begin{eqnarray*}
\mu_{xy}(g) & = & \mu(g(xy)^{-1})=\mu(g\cdot x^{-1}y^{-1})\quad\forall\, g\in G;\,\mbox{and}\\
\mu_{x_{\circ}y_{\circ}}(g) & = & \mu(g(x_{\circ}y_{\circ})^{-1})=\mu(g\cdot x_{\circ}^{-1}y_{\circ}^{-1})\quad\forall\, g\in G.
\end{eqnarray*}
Now $\forall\, x,y\in G$,
\begin{eqnarray*}
\mu(g\cdot x^{-1}y^{-1}) & = & \mu[e(g\cdot x^{-1}y^{-1})]\\
 & = & \mu[((x_{\circ}y_{\circ})^{-1}(x_{\circ}y_{\circ}))(g\cdot x^{-1}y^{-1})]\\
 & = & \mu[((x_{\circ}^{-1}y_{\circ}^{-1})(x_{\circ}y_{\circ}))(g\cdot x^{-1}y^{-1})]\\
 & = & \mu[g(((x_{\circ}^{-1}y_{\circ}^{-1})(x_{\circ}y_{\circ}))(x^{-1}y^{-1}))]\mbox{\,\,[in \ensuremath{G}; \ensuremath{a(bc)=b(ac)\,}[8]]}\\
 & = & \mu[g(((x^{-1}y^{-1})(x_{\circ}y_{\circ}))(x_{\circ}^{-1}y_{\circ}^{-1}))]\,\,[\mbox{using left invertive law]}\\
 & = & \mu[((x^{-1}y^{-1})(x_{\circ}y_{\circ}))(g(x_{\circ}^{-1}y_{\circ}^{-1}))]\mbox{\,\,[in \ensuremath{G;}\ensuremath{\, a(bc)=b(ac)\,}[8]]}\\
 & \geq & \mu((x^{-1}y^{-1})(x_{\circ}y_{\circ}))\wedge\mu(g(x_{\circ}^{-1}y_{\circ}^{-1}))
\end{eqnarray*}
\begin{eqnarray}
\Rightarrow\mu(g\cdot x^{-1}y^{-1}) & \geq & \mu((x^{-1}y^{-1})(x_{\circ}y_{\circ}))\wedge\mu(g(x_{\circ}^{-1}y_{\circ}^{-1}))\label{p}
\end{eqnarray}

\noindent Now we show that $\mu((x^{-1}y^{-1})(x_{\circ}y_{\circ}))=\mu(e)$
in (\ref{p}) . Let $\mu_{x}=\mu_{x_{\circ}}\Rightarrow\mu_{x}(g)=\mu_{x_{\circ}}(g)\,\,\forall\, g\in G$,
\begin{eqnarray}
\Rightarrow\mu(gx^{-1}) & = & \mu(gx_{\circ}^{-1}).\qquad[\mbox{by }(\ref{eq:1a-1})]\label{eq:2a}
\end{eqnarray}
Similarly, since $\mu_{y}=\mu_{y_{\circ}}\Rightarrow\mu_{y}(g)=\mu_{y_{\circ}}(g)\,\,\forall\, g\in G,$
\begin{eqnarray}
\Rightarrow\mu(gy^{-1}) & = & \mu(gy_{\circ}^{-1}).\qquad[\mbox{by }(\ref{eq:1a-1})]\label{eq:3a}
\end{eqnarray}
Now,
\begin{eqnarray*}
\mu((x^{-1}y^{-1})(x_{\circ}y_{\circ})) & = & \mu((x_{\circ}y_{\circ}\cdot y^{-1})x^{-1})\,\,[\mbox{using left invertive law]}\\
 & = & \mu((x_{\circ}y_{\circ}\cdot y^{-1})x_{\circ}^{-1})\\
 &  & [\mbox{substituting \,\ensuremath{g}\,\ by }(x_{\circ}y_{\circ}\cdot y^{-1})\mbox{ in }(\ref{eq:2a})]\\
 & = & \mu(x_{\circ}(y_{\circ}y^{-1}\cdot x_{\circ}^{-1}))\,\,\mbox{[in \ensuremath{G}; \ensuremath{(ab\cdot c)d=a(bc\cdot d)\,}[8]]}\\
 & = & \mu((y_{\circ}y^{-1})(x_{\circ}x_{\circ}^{-1}))\mbox{\,\,[in \ensuremath{G}; \ensuremath{\, a(bc)=b(ac)\,}[8]]}\\
 & = & \mu(y_{\circ}y^{-1}\cdot e)=\mu(ey^{-1}\cdot y_{\circ})\,\,[\mbox{using left invertive law]}\\
 & = & \mu(y^{-1}y_{\circ})\\
 & = & \mu(y_{\circ}y^{-1})\qquad\qquad\,\,[\mbox{by Proposition }\ref{b-1}]\\
 & = & \mu(y_{\circ}y_{\circ}^{-1})\qquad\qquad\,\,[\mbox{substituting \ensuremath{\, g\,}by }y_{\circ}\mbox{ in }(\ref{eq:3a})]\\
 & = & \mu(e).
\end{eqnarray*}
Therefore, (\ref{p}) implies that $\mu(g\cdot x^{-1}y^{-1})\geq\mu(g\cdot x_{\circ}^{-1}y_{\circ}^{-1})$.
Similarly one can prove that $\mu(g\cdot x_{\circ}^{-1}y_{\circ}^{-1})\geq\mu(g\cdot x^{-1}y^{-1})$.
Consequently,
\begin{eqnarray*}
\mu(g\cdot x^{-1}y^{-1}) & = & \mu(g\cdot x_{\circ}^{-1}y_{\circ}^{-1})\\
\Rightarrow\mu(g\cdot(xy)^{-1}) & = & \mu(g\cdot(x_{\circ}y_{\circ})^{-1})\\
\Rightarrow\mu_{xy}(g) & = & \mu_{x_{\circ}y_{\circ}}(g)\,\,\forall\, g\in G\\
\Rightarrow\mu_{xy} & = & \mu_{x_{\circ}y_{\circ}}.
\end{eqnarray*}
Hence the product of cosets is well-defined. Now we show that $G\diagup\mu$
form an AG-group under the operation $\circ$.

\noindent $G\diagup\mu$ is closed under the operation $\circ$. Also
$G\diagup\mu$ satisfies left invertive law under $\circ$ ; since
$(\mu_{x}\circ\mu_{y})\circ\mu_{z}=\mu_{xy}\circ\mu_{z}=\mu_{xy\cdot z}=\mu_{zy\cdot x}=\mu_{zy}\circ\mu_{x}=(\mu_{z}\circ\mu_{y})\circ\mu_{x}$
$\forall\, x,y,z\in G$. Now for any $x\in G$, $(\mu_{e}\circ\mu_{x})(g)=(\mu_{ex})(g)=(\mu_{x})(g)\Rightarrow(\mu_{e}\circ\mu_{x})=\mu_{x}\,\,\forall\, g\,\in G$,
but $(\mu_{x}\circ\mu_{e})(g)=(\mu_{xe})(g)\neq(\mu_{x})(g)\Rightarrow(\mu_{x}\circ\mu_{e})\neq\mu_{x}\,\,\forall g\,\in G$.
This implies that $\mu_{e}$ is the left identity of $G\diagup\mu$.
As an AG-group $G$ is non associative therefore, $(\mu_{x}\circ\mu_{y})\circ\mu_{z}\neq\mu_{x}\circ(\mu_{y}\circ\mu_{z})$.
Finally, $\forall\, x\in G$, once $(\mu_{x}\circ\mu_{x^{-1}})(g)=(\mu_{xx^{-1}})(g)=(\mu_{e})(g)\Rightarrow\mu_{x}\circ\mu_{x^{-1}}=\mu_{e}\,\,\forall\, g\in G$,
and $(\mu_{x^{-1}}\circ\mu_{x})(g)=(\mu_{x^{-1}x})(g)=(\mu_{e})(g)\Rightarrow\mu_{x^{-1}}\circ\mu_{x}=\mu_{e}\,\,\forall\, g\in G$
the inverse of each $\mu_{x}$ exists and is $\mu_{x^{-1}}$. Hence
it follows that $G\diagup\mu$\emph{ }is an AG-subgroup. \end {proof}

\noindent \begin {remark} The AG-group $G\diagup\mu$ defined in
Theorem \ref{s} is called quotient AG-group of $G$ relative to the
normal fuzzy AG-subgroup $\mu$.

\noindent \end {remark}

\noindent \begin {theorem} \label{FC}Let $\nu\in F(G)$ and $H$
be any AG-subgroup of $G$. Define $\xi\in FP(G\diagup H)$ as follows:
\begin{eqnarray*}
\xi(Hx) & = & \vee\{\nu(z)\,:\, z\in Hx\}\,\,\,\forall x\in G.
\end{eqnarray*}
Then $\xi\in F(G\diagup H)$. \end {theorem}

\noindent \begin {proof} Since $\forall\, x,y\in G$,
\begin{eqnarray*}
\xi(HxHy)=\xi(H(xy)) & = & \vee\{\nu(z)\,:\, z\in H(xy)\}\\
 & = & \vee\{\nu(uv)\,:\, u\in Hx,\, v\in Hy\}\\
 & \geq & \vee\{\nu(u)\wedge\nu(v)\,:\, u\in Hx,\, v\in Hy\}\\
 & = & (\vee\{\nu(u)\,:\, u\in Hx\})\wedge(\vee\{\nu(v)\,:\, v\in Hy\})\\
 & = & \xi(Hx)\wedge\xi(Hy)
\end{eqnarray*}
and $\forall\, x\in G$,
\begin{eqnarray*}
\xi(Hx)^{-1}=\xi(Hx^{-1}) & = & \vee\{\nu(z)\,:\, z\in Hx^{-1}\}\qquad\qquad\qquad\qquad\\
 & = & \vee\{\nu(w^{-1})\,:\, w^{-1}\in Hx^{-1}\}\qquad\qquad\qquad\\
 & \geq & \vee\{\nu(w)\,:\, w\in Hx\}\qquad\qquad\qquad\\
 & = & \xi(Hx).
\end{eqnarray*}
Hence $\xi\in F(G\diagup H)$. \end {proof}

\noindent \begin {remark} The fuzzy AG-subgroup defined in Theorem
\ref{FC} is called Quotient fuzzy AG-subgroup or factor fuzzy AG-subgroup
of $G$, and is denoted by $\nu\diagup H$.

\noindent \end {remark}

\noindent \begin {theorem} \label{c}Let $\mu\in\mathcal{F}(G)$.
Then $\mu_{x}=\mu_{y}\Leftrightarrow\underset{x}{\mu_{*}}=\underset{y}{\mu_{*}}\,\,\forall\, x,y\in G$.\end {theorem}

\noindent \begin {proof} Let $\mu_{x}=\mu_{y}$, then
\begin{eqnarray}
\mu_{x}(g) & = & \mu_{y}(g)\,\,\forall\, g\in G\nonumber \\
\Rightarrow\mu(gx^{-1}) & = & \mu(gy^{-1})\,\,\forall\, g\in G\quad[\mbox{using (\ref{eq:1a-1})}]\label{eq:1a}
\end{eqnarray}
put $g=y$, in (\ref{eq:1a}) we get $\mu(yx^{-1})=\mu(e)$ $\Rightarrow yx^{-1}\in\mu_{*}.\,\,\,[\mbox{by (\ref{a1})}]$
$\Rightarrow$$(yx^{-1})x\in\underset{x}{\mu_{*}}\Rightarrow(xx^{-1})y\in\underset{x}{\mu_{*}}$
$(\mbox{using left invertive law})$ $\Rightarrow y\in\underset{x}{\mu_{*}}$,
but $y\in\underset{y}{\mu_{*}}$, therefore, $\underset{y}{\mu_{*}}\subseteq\underset{x}{\mu_{*}}$.

\noindent Again put $g=x$, in (\ref{eq:1a}) we get $\mu(xx^{-1})=\mu(xy^{-1})\Rightarrow\mu(xy^{-1})=\mu(e)\Rightarrow xy^{-1}\in\mu_{*}.\,\,\,[\mbox{by (\ref{a1})}]\Rightarrow(xy^{-1})y\in\underset{y}{\mu_{*}}\Rightarrow x\in\underset{y}{\mu_{*}}$,
but $x\in\underset{x}{\mu_{*}}$. This implies that, $\underset{x}{\mu_{*}}\subseteq\underset{y}{\mu_{*}}$.
Thus $\underset{x}{\mu_{*}}\subseteq\underset{y}{\mu_{*}}\subseteq\underset{x}{\mu_{*}}$.
Hence $\underset{x}{\mu_{*}}=\underset{y}{\mu_{*}}.$

\noindent Conversely; let $\underset{x}{\mu_{*}}=\underset{y}{\mu_{*}}\Rightarrow\underset{x}{\mu_{*}}\circ\underset{y^{-1}}{\mu_{*}}=\underset{y}{\mu_{*}}\circ\underset{y^{-1}}{\mu_{*}}\Rightarrow\underset{xy^{-1}}{\mu_{*}}=\underset{yy^{-1}}{\mu_{*}}\Rightarrow\underset{xy^{-1}}{\mu_{*}}=\underset{e}{\mu_{*}}=\mu_{*}\Rightarrow xy^{-1}\in\mu_{*}$.
Now for any $x,y\in G$ it follows that
\begin{eqnarray*}
\mu(gx^{-1}) & = & \mu(g((y^{-1}y)x^{-1}))\\
 & = & \mu(g((x^{-1}y)y^{-1}))\qquad[\mbox{using left invertive law]}\\
 & = & \mu((x^{-1}y)(gy^{-1}))\qquad\mbox{[in \ensuremath{G}; \ensuremath{a(bc)=b(ac)}\,[8]]}\\
 & \geq & \mu(x^{-1}y)\wedge\mu(gy^{-1})\quad[\mu\in\mathcal{F}(G)]\\
 & = & \mu((xy^{-1})^{-1})\wedge\mu(gy^{-1})\\
 & = & \mu(xy^{-1})\wedge\mu(gy^{-1})\quad\,[\mbox{by Lemma \ref{ad}-(ii)}]\\
 & = & \mu(e)\wedge\mu(gy^{-1})\qquad\quad[\mbox{by }(\ref{a1});\mbox{ as }xy^{-1}\in\mu_{*}]\\
 & = & \mu(gy^{-1})\qquad\qquad\quad\quad[\mbox{by Lemma \ref{ad}-(i)}]
\end{eqnarray*}
This implies that $\mu(gx^{-1})\geq\mu(gy^{-1})$.

\noindent By similar arrangements we can show that, $\mu(gy^{-1})\geq\mu(gx^{-1})$.
Consequently $\mu(gx^{-1})=\mu(gy^{-1})\Rightarrow\mu_{x}(g)=\mu_{y}(g)\,\,\forall\, g\in G$
by (\ref{eq:1a-1}). Hence $\mu_{x}=\mu_{y}$. \end {proof} 

\noindent \begin {theorem} \label{aa}Let $\mu\in NF(G)$ and $\mu_{x}=\mu_{y}$,
then $\mu(x)=\mu(y)$ $\forall\, x,y\in G$. \end {theorem}

\noindent \begin {proof} Let $x,y\in G$, then $\mu_{x}=\mu_{y}\Leftrightarrow\underset{x}{\mu_{*}}=\underset{y}{\mu_{*}}\Rightarrow\underset{x}{\mu_{*}}\circ\underset{y^{-1}}{\mu_{*}}=\underset{y}{\mu_{*}}\circ\underset{y^{-1}}{\mu_{*}}\Rightarrow\underset{xy^{-1}}{\mu_{*}}=\underset{yy^{-1}}{\mu_{*}}\Rightarrow\underset{xy^{-1}}{\mu_{*}}=\underset{e}{\mu_{*}}=\mu_{*}\Rightarrow xy^{-1}\in\mu_{*}$;
(using Theorem \ref{c} and the definition of fuzzy cosets). Therefore,
\begin{eqnarray*}
\mu(y)=\mu(y^{-1}) & = & \mu(x^{-1}y^{-1}\cdot(x^{-1})^{-1})\quad[\mu\in NF(G)]\\
 & = & \mu(x^{-1}y^{-1}\cdot x)\\
 & = & \mu(xy^{-1}\cdot x^{-1})\qquad[\mbox{using left invertive law]}\\
 & \geq & \mu(xy^{-1})\wedge\mu(x^{-1})\\
 & = & \mu(e)\wedge\mu(x)\qquad\,\,\,\,[\mbox{ by }(\ref{a1}),\mbox{ as }xy^{-1}\in\mu_{*}]\\
 & = & \mu(x)\qquad\qquad\quad\,\,\,\,[\mbox{by Lemma \ref{ad}-(i)}]\\
\Rightarrow\mu(y) & \geq & \mu(x).
\end{eqnarray*}

\noindent Similarly, we can show that $\mu(x)\geq\mu(y)$. This implies
that $\mu(x)\geq\mu(y)\geq\mu(x).$ Hence $\mu(x)=\mu(y)$. \end {proof}

\noindent \begin {proposition} Let $\mu\in NF(G)$. Then $\mu_{x}(xg)=\mu_{x}(gx)=\mu(g)\,\,\forall\, g\in G$.\end {proposition}

\noindent \begin {proof} Using definition of cosets of fuzzy AG-subgroup,
it follows that for $g\in G$; $\mu_{x}(xg)=\mu(xg\cdot x^{-1})=\mu(g)$.
And $\mu_{x}(gx)=\mu(gx\cdot x^{-1})=\mu(x^{-1}x\cdot g)=\mu(eg)=\mu(g)$.
Hence $\mu_{x}(xg)=\mu_{x}(gx)=\mu(g)\,\,\forall\, g\in G$. \end {proof}

\noindent \begin {theorem}\label{d}Let $\mu\in NF(G)$. Then the
following assertions hold:

\begin{enumerate}[(i)]

\item $G\diagup\mu\cong G\diagup\mu_{*}$;

\item If\emph{ }$\nu\in FP(G\diagup\mu)$; defined by $\nu(\mu_{x})=\mu(x)\,\,\forall\, x\in G$.
Then $\nu\in NF(G\diagup\mu)$.

\end{enumerate}

\noindent \end {theorem}

\noindent \begin {proof}  As both $G\diagup\mu$ and $G\diagup\mu_{*}$
are AG-groups by Theorem \ref{s} and $\phi:G\diagup\mu\rightarrow G\diagup\mu_{*}$
given by $\phi(\mu_{x})=\underset{x}{\mu_{*}}\,\,\forall\, x\in G$
is an isomorphism by Theorem \ref{c} and the fact that $\mu_{x}\circ\mu_{y}=\mu_{xy}$
and $\underset{x}{\mu_{*}}\circ\underset{y}{\mu_{*}}=\underset{xy}{\mu_{*}}$.

(ii) Let\emph{ }$\nu\in FP(G\diagup\mu)$, be defined by $\nu(\mu_{x})=\mu(x)\,\,\forall\, x\in G$.
We show that $\nu\in NF(G\diagup\mu)$. Since
\begin{eqnarray*}
\nu(\mu_{x}\circ\mu_{y}) & = & \nu(\mu_{xy})\\
 & = & \mu(xy)\qquad\,\,\qquad[\mbox{by definition of \ensuremath{\nu}}]\\
 & \geq & \mu(x)\wedge\mu(y)\,\,\,\qquad\,\,[\mu\in NF(G)]\\
 & = & \nu(\mu_{x})\wedge\nu(\mu_{y})\quad\,[\mbox{by definition of \ensuremath{\nu}}]
\end{eqnarray*}
$\forall\, x,y\in G$, and
\begin{eqnarray*}
\nu((\mu_{x})^{-1}) & = & \nu(\mu_{x^{-1}})=\mu(x^{-1})\geq\mu(x)=\nu(\mu_{x})\qquad\,\,
\end{eqnarray*}
$\forall\, x\in G$. Hence $\nu\in F(G\diagup\mu)$. Further, since
\begin{eqnarray*}
\nu((\mu_{x}\circ\mu_{y})\circ(\mu_{x})^{-1}) & = & \nu(\mu_{xy}\circ\mu_{x^{-1}})\\
 & = & \nu(\mu_{xy\cdot x^{-1}})\\
 & = & \mu(xy\cdot x^{-1})\,\,\quad\,[\mbox{by definition of \ensuremath{\nu}}]\\
 & = & \mu(y)\qquad\,\,\,\,\,\,\qquad[\mu\in NF(G)]\\
 & = & \nu(\mu_{y}).\qquad\qquad[\mbox{by definition of \ensuremath{\nu}}]
\end{eqnarray*}
$\forall\, x,y\in G$. Hence $\nu\in NF(G\diagup\mu)$. \end {proof}

\noindent \begin {theorem}\label{e}Let $\mu\in NF(G)$. Define a
mapping $\theta:G\rightarrow G\diagup\mu$ as follows:
\begin{eqnarray}
\theta(x) & = & \mu_{x}\,\,\forall\, x\in G.\label{eq:4a}
\end{eqnarray}
Then $\theta$ is homomorphism with kernel $\mu_{*}$. \end {theorem}

\noindent \begin {proof} Since 
\begin{eqnarray*}
\theta(xy) & = & \mu_{xy}=\mu_{x}\circ\mu_{y}=\theta(x)\theta(y)\,\,\,\forall\, x,y\in G.
\end{eqnarray*}
Hence $\theta$ is homomorphism. Further, the kernel of $\theta$
consists of all $x\in G$ for which $\mu_{x}=\mu_{e}\Leftrightarrow\mu(x)=\mu(e)$,
(by Theorem \ref{aa}) $\Leftrightarrow x\in\mu_{*}$. Thus $Ker\theta=\mu_{*}$.
\end {proof}

\noindent \begin {theorem}\emph{ }Let $\mu\in NF(G)$, and $G\diagup\mu$
is an AG-group. Then each $\zeta\in NF(G\diagup\mu)$ corresponds
in a natural way to $\nu\in NF(G)$.\end {theorem}

\noindent \begin {proof} Let $\zeta\in NF(G\diagup\mu)$. Define
a mapping $\nu:G\rightarrow[0,1]$ as follows:
\begin{eqnarray*}
\nu(x) & = & \zeta(\mu_{x})\,\,\forall\, x\in G.
\end{eqnarray*}
First we show that $\nu\in F(G)$. Since $\forall\, x,y\in G$,
\begin{eqnarray*}
\nu(xy) & = & \zeta(\mu_{xy})\\
 & = & \zeta(\mu_{x}\circ\mu_{y})\\
 & \geq & \zeta(\mu_{x})\wedge\zeta(\mu_{y})\qquad[\zeta\in NF(G\diagup\mu)]\\
 & = & \nu(x)\wedge\nu(y)
\end{eqnarray*}
and $\forall\, x\in G$,
\begin{eqnarray*}
\nu(x^{-1}) & = & \zeta(\mu_{x^{-1}})=\zeta(\mu_{x})^{-1}\geq\zeta(\mu_{x})=\nu(x)\qquad\qquad\qquad
\end{eqnarray*}
Thus $\nu\in F(G)$.

\noindent Further, since $\forall\, x,y\in G$,
\begin{eqnarray*}
\nu(xy\cdot x^{-1}) & = & \zeta(\mu_{xy\cdot x^{-1}})\\
 & = & \zeta(\mu_{y})\qquad[\mu\in NF(G)]\\
 & = & \nu(y).
\end{eqnarray*}
Hence $\nu\in NF(G)$. \end {proof}

\noindent In the following we introduce fuzzy Lagrange's Theorem for
AG-group of finite order. We start with the following definition.

\noindent \begin {definition} Let $G$ be a finite AG-group, $\mu\in F(G)$
and $G\diagup\mu$ is an AG-group. Then the cardinality of $G\diagup\mu$
is called the index of fuzzy AG-subgroup of $\mu$ in $G$ written
as $[G:\mu]$.\end {definition}

\noindent \begin {theorem} (Fuzzy Lagrange's Theorem for AG-subgroup).
Let $G$ be a finite AG-group, $\mu\in F(G)$. Then the index of fuzzy
AG-subgroup of $\mu$ divides the order of $G$.\end {theorem}\begin {proof}
It follows from Theorem \ref{e}, that there is homomorphism $\theta$
from $G$ into $G\diagup\mu$, the set of all fuzzy cosets of $\mu$,
defined in (\ref{eq:4a}). Let $H$ be an AG-subgroup of $G$ defined
by $H=\{h\in G\,:\,\mu_{h}=\mu_{e}\}$. Let $h\in H$, then $\mu_{h}=\mu_{e}\Leftrightarrow\underset{h}{\mu_{*}}=\underset{e}{\mu_{*}}$
using Theorem \ref{c}. Therefore, $H=\{h\in G\,:\,\underset{h}{\mu_{*}}=\underset{e}{\mu_{*}}\}$.
Now decomposing $G$ as a disjoint union of the cosets of $G$ with
respect to $H$ i.e.
\begin{eqnarray}
G & = & (H=Hx_{1})\cup Hx_{2}\cup\cdots\cup Hx_{k},\label{eq:5a}
\end{eqnarray}
where $x_{1}\in H$ and $x_{i}\in G;\,1\leq i\leq k$. Now, we show
that corresponding to each cost $Hx_{i};\,\,1\leq i\leq k$, given
in (\ref{eq:5a}) there is a fuzzy coset belonging to $G\diagup\mu$,
and further this correspondence is one-one. To see this, consider
any coset $Hx_{i}$ for any $h\in H$, we have that; $\theta(hx_{i})=\mu_{hx_{i}}=\mu_{h}\circ\mu_{x_{i}}=\mu_{e}\cdot\mu_{x_{i}}=\mu_{ex_{i}}=\mu_{x_{i}}$.
Thus $\theta$ maps each element of $Hx_{i}$ into the fuzzy cosets
$\mu_{x_{i}}$.

\noindent Now we show that $\theta$ is well-defined. Let $Hx_{i}=Hx_{j}$$\ensuremath{\mbox{ , for each \ensuremath{i,j:}\,}1\leq i\leq k}\,\ \textrm{and }\ensuremath{1\leq j\leq k}.$
Then
\begin{eqnarray*}
x_{j}^{-1}x_{i}\in H &  & \qquad[\mbox{cosets in AG-groups}]\\
\Rightarrow\mu_{x_{j}^{-1}x_{i}} & = & \mu_{e}\\
\Rightarrow\underset{(x_{j}^{-1}x_{i})}{\mu_{*}} & = & \underset{e}{\mu_{*}}\qquad[\mbox{by Theorem \ref{c}]}\\
\Rightarrow\underset{(x_{j}^{-1}x_{i})}{\mu_{*}}\circ\underset{x_{i}^{-1}}{\mu_{*}} & = & \underset{e}{\mu_{*}}\circ\underset{x_{i}^{-1}}{\mu_{*}}\\
\Rightarrow\underset{(x_{j}^{-1}x_{i})x_{i}^{-1}}{\mu_{*}} & = & \underset{ex_{i}^{-1}}{\mu_{*}}
\end{eqnarray*}

\noindent 
\begin{eqnarray*}
\Rightarrow\underset{(x_{i}^{-1}x_{i})x_{j}^{-1}}{\mu_{*}} & = & \underset{ex_{i}^{-1}}{\mu_{*}}\qquad[\mbox{using left invertive law}]\\
\Rightarrow\underset{x_{i}^{-1}}{\mu_{*}} & = & \underset{x_{j}^{-1}}{\mu_{*}}\\
\Rightarrow\underset{x_{i}}{\mu_{*}} & = & \underset{x_{j}}{\mu_{*}}\\
\Rightarrow\mu_{x_{i}} & = & \mu_{x_{j}}\,\qquad[\mbox{by Theorem \ref{c}]}\\
\Rightarrow\theta(Hx_{i}) & = & \theta(Hx_{j}).
\end{eqnarray*}
Thus $\theta$ is well-defined.

\noindent Further, we show that $\theta$ is one-one; for each $i,\, j$
where $1\leq i\leq k$ and $1\leq j\leq k$; assume that
\begin{eqnarray*}
\theta(Hx_{i}) & = & \theta(Hx_{j})\\
\Rightarrow\mu_{x_{i}} & = & \mu_{x_{j}}\\
\Rightarrow\underset{x_{i}}{\mu_{*}} & = & \underset{x_{j}}{\mu_{*}}\qquad\qquad[\mbox{ Theorem}\ref{c}]\\
\Rightarrow\underset{x_{i}^{-1}}{\mu_{*}} & = & \underset{x_{j}^{-1}}{\mu_{*}}\qquad[\mbox{cosets in AG-groups }]\\
\Rightarrow\underset{e}{\mu_{*}}\circ\underset{x_{i}^{-1}}{\mu_{*}} & = & \underset{e}{\mu_{*}}\circ\underset{x_{j}^{-1}}{\mu_{*}}
\end{eqnarray*}
\begin{eqnarray*}
\Rightarrow\underset{ex_{i}^{-1}}{\mu_{*}} & = & \underset{ex_{j}^{-1}\textrm{}}{\mu_{*}}\textrm{\qquad\qquad\qquad\qquad\qquad\qquad\qquad}\\
\Rightarrow\underset{ex_{i}^{-1}}{\mu_{*}} & = & \underset{(x_{i}^{-1}x_{i}\cdot x_{j}^{-1})}{\mu_{*}}\\
\Rightarrow\underset{ex_{i}^{-1}}{\mu_{*}} & = & \underset{(x_{j}^{-1}x_{i}\cdot x_{i}^{-1})}{\mu_{*}}\\
\Rightarrow\underset{(x_{j}^{-1}x_{i})}{\mu_{*}}\circ\underset{x_{i}^{-1}}{\mu_{*}} & = & \underset{e}{\mu_{*}}\circ\underset{x_{i}^{-1}}{\mu_{*}}
\end{eqnarray*}
\begin{eqnarray*}
\Rightarrow\underset{(x_{j}^{-1}x_{i})}{\mu_{*}} & = & \underset{e}{\mu_{*}}\\
\Rightarrow\mu_{x_{j}^{-1}x_{i}} & = & \mu_{e}\qquad[\mbox{by Theorem }\ref{c}]\\
\Rightarrow x_{j}^{-1}x_{i}\in H
\end{eqnarray*}
\begin{eqnarray*}
\Leftrightarrow Hx_{i} & = & Hx_{j}\qquad\mbox{\ensuremath{\mbox{[for each \ensuremath{i\,}and \ensuremath{j\,}where \,\,}1\leq i\leq k\,}and \ensuremath{1\leq j\leq k.}]}
\end{eqnarray*}
From above discussion it is now clear that the number of distinct
cosets of $H$ (index) in $G$ equals the number of fuzzy cosets of
$\mu$, which is a divisor of the order of $G$. Hence we conclude
that the index of $\mu$ also divides the order of $G$. \end {proof}


\begin{thebibliography}{10}
\bibitem{Zad} L. A. Zadeh, \emph{Fuzzy Sets.} Information and Control,\emph{
}8(3), 338-353, (1965).

\bibitem{Ros} A. Rosenfeld, \emph{Fuzzy group}, J. Math. Anal. Appl.
35, 512-517, (1971). 

\bibitem{Anto} J. M. Anthony and H. Sherwood, \emph{Fuzzy groups
redefined}, J. Math. Anal. Appl. 69 (1979) 124-130.

\bibitem{Meng} M. Daoji, \emph{Fuzzy groups}, Fuzzy Math. 2 (2) (1982)
49-60.

\bibitem{Wu} Wangming Wu, \emph{Fuzzy congruences and normal fuzzy
subgroups}, Fuzzy Math. 3 (1988) 9-20.

\bibitem{Wan}Wangming Wu, \emph{Normal fuzzy subgroups}, Fuzzy Math
1, 21-30, (1981).

\bibitem{Kamran} M. Kamran. \emph{Conditions for LA-semigroups to
resemble associative structures}. PhD Thesis, Quaid-i-Azam University,
Islamabad, Pakistan, (1993).

\bibitem{Sha Thes} M. Shah, \emph{A Theoretical and Computational
Investigation of AG-groups}, PhD Thesis, Quaid-i-Azam University,
Islamabad, Pakistan.

\bibitem{idealsinAG}\textsc{ }Q. Mushtaq, Madad Khan,\emph{ Ideals
in the left almost semigroups,} Proceedings of 4th International Pure
Mathematics Conference,\textbf{ }65-77 (2003). 

\bibitem{I. Ahmad} I. Ahmad, Amanullah, M. Shah, \emph{Fuzzy AG-Subgroup}s,
Life Sciences Journal 9(4), 3931-3936, (2012).

\bibitem{key-1} Amanullah, I. Ahmad, M. Shah, On the Equal-height
Elements of Fuzzy AG-subgroups. Life Science Journal, 10(4): 3143-3146
(2013).

\bibitem{SIA1} M. Shah, I. Ahmad and A. Ali ``Discovery of new classes
of AG-groupoids'' \emph{Research Journal of Recent Sciences}, 1(11),
47, 2012. 

\bibitem{MshahT} M. Shah and A. Ali `` Some structural properties
of AG-groups'' \emph{Int. Math. Forum,} 7, 1661-1667, 2011.

\bibitem{intro2}M. Rashad, I. Ahmad, and M. Shah, Left Transitive
AG-groupoids, www.arxiv.org, 1402.5296

\bibitem{key-1-13}M. Shah, I. Ahmad and A. Ali ``On Introduction
of New Classes of AG-groupoids'' \emph{Research Journal of Recent
Sciences,} 2(1), 67-70, 2013

\bibitem{RAAS4}M. Rashad, I. Ahmad, Amanullah and M. Shah, On relations
between right alternative and nuclear square AG-groupoids, Int. Mathematical
Forum, Vol. 8(5), 237-243, 2013.

\bibitem{AGst5} I. Ahmad, M. Rashad and M. Shah, Some Properties
of AG{*}-groupoid, Res. J. Recent Sci.,Vol. 2(4), 91-93, April (2013).

\bibitem{AGst-6}I. Ahmad, M. Rashad and M. Shah, Some new result
on T$^{1}$, T$^{2}$ and T$^{4}$-AG-groupoids, Research Journal
of Recent Sciences, Vol. 2(3), 64-66, (2013).

\bibitem{mod}Amanullah, M. Rashad, I. Ahmad, M. Shah, On Modulo AG-groupoids,
arXiv: 1403.2564 

\bibitem{mod m} Amanullah, M. Rashad, I. Ahmad, M. Shah, M. Yousaf,
Modulo Matrix AG-groupoids and Modulo AG-groups arXiv:1403.2304\end{thebibliography}
\end{document}